\documentclass[11pt]{amsart}  
\usepackage{amsmath,amssymb,amsfonts,times,hyperref}
\usepackage{graphicx}

\newtheorem{defi}{Definition}[section]
\newtheorem{definition}[defi]{Definition}

\newtheorem{theorem}[defi]{Theorem}

\newtheorem{remark}[defi]{Remark}

\newtheorem{lemma}[defi]{Lemma}

\newcommand{\N}{{\mathbb{N}}} 
\newcommand{\Z}{{\mathbb{Z}}} 

\newcommand{\R}{{\mathbb{R}}} 

\newcommand{\e}{{\mathrm{e}}}
\newcommand{\G}{\mathcal{G}}
\newcommand{\CC}{\mathcal{C}}
\newcommand{\M}{\mathcal{M}}
\newcommand{\vol}{\operatorname{vol}}
\newcommand{\sech}{\operatorname{sech}}
\newcommand{\Supp}{\operatorname{supp}}

\newcommand{\1}{{\bf \operatorname{1}}}

\begin{document}

\title{The density of sets avoiding distance $1$ in Euclidean space}

\author{Christine Bachoc}
\address{Christine Bachoc \\Univ. Bordeaux, Institut de math\'ematiques de Bordeaux\\ 351, cours de la Lib\'eration\\ 33405 Talence\\ France}
\email{christine.bachoc@u-bordeaux.fr}
\thanks{This study has been carried out with financial support from the French State, managed by the French National Research Agency (ANR) in the frame of the Investments for the future Programme IdEx Bordeaux (ANR-10-IDEX-03-02).}

\author{Alberto Passuello}
\address{Alberto Passuello\\Univ. Bordeaux, Institut de math\'ematiques de Bordeaux\\ 351, cours de la Lib\'eration\\ 33405 Talence\\ France}
\email{alberto.passuello@u-bordeaux.fr}

\author{Alain Thiery}
\address{Alain Thiery \\Univ. Bordeaux, Institut de math\'ematiques de Bordeaux\\ 351, cours de la Lib\'eration\\ 33405 Talence\\ France}
\email{alain.thiery@u-bordeaux.fr}

\date{\today}


\begin{abstract}
We improve by an exponential factor the best known asymptotic  upper bound for the density of sets avoiding $1$ in
Euclidean space. This result is obtained by a combination of an
analytic bound that is an analogue of Lov\'asz theta number and of a
combinatorial argument  involving finite subgraphs of the unit
distance graph. In turn, we straightforwardly obtain an asymptotic
improvement for the measurable chromatic number of Euclidean
space. We also tighten previous results for  the dimensions between $4$ and $24$.
\end{abstract}

\maketitle

\section{Introduction}

In the Euclidean space $\R^n$, a subset $S$ is said to \emph{avoid $1$} if
$\|x-y\|\neq 1$ for all $x,y$ in $S$. For example, one can take the union of open
balls of radius $1/2$ with centers in $(2\Z)^n$. It is natural to
wonder how large $S$ can be, given that it avoids $1$. To be more precise, 
if $S$ is 
 a Lebesgue measurable set, its \emph{density}  $\delta(S)$ is defined  by 
\begin{equation*}
\delta(S)=\limsup_{R\to \infty} \frac{\vol([-R,R]^n\cap
  S)}{\vol([-R,R]^n)},
\end{equation*}
where $\vol(S)$ is the Lebesgue measure of $S$.  We are interested in
the supreme 
density  $m_1(\R^n)$ of the Lebesgue measurable sets avoiding $1$.

In terms of graphs, 
a set  $S$ avoiding $1$  is an independent
set of the \emph{unit distance graph}, the graph drawn on $\R^n$ that
connects by an edge  every pair of points at distance $1$, and
$m_1(\R^n)$ is a substitute for the
\emph{independence number} of this graph.

Larman and Rogers introduced in \cite{LR} the number  $m_1(\R^n)$ in order to allow
for analytic tools in the study of the \emph{chromatic number
  $\chi(\R^n)$} of the unit distance graph, i.e. the minimal number of colors
needed to color $\R^n$ so that points at distance $1$ receive
different colors. Indeed, the inequality 
\begin{equation}\label{e1}
\chi_m(\R^n)\geq \frac{1}{m_1(\R^n)}.
\end{equation}
holds, where $\chi_m(\R^n)$ denotes the 
 \emph{measurable chromatic number}  of $\R^n$. In the definition of
 $\chi_m(\R^n)$,
the measurability of the color classes is required, so
$\chi_m(\R^n)\geq \chi(\R^n)$. We note that  \eqref{e1} is the exact analogue of the well known
relation between the chromatic number $\chi(G)$ and the
independence number $\alpha(G)$ of a finite graph $G=(V,E)$:
\begin{equation*}
\chi(G)\geq \frac{|V|}{\alpha(G)}.
\end{equation*}

Following \eqref{e1}, in order to lower bound
$\chi_m(\R^n)$, it is enough to upper bound $m_1(\R^n)$.
As shown in \cite{LR}, finite configurations of points in $\R^n$ can be
used for this purpose. Indeed, if $G=(V,E)$ is a finite  induced subgraph of the unit distance graph  of $\R^n$, meaning that
$V=\{v_1,\dots,v_M\}\subset \R^n$ and $E=\{ \{i,j\}: \|v_i-v_j\|=1\}$, then
\begin{equation}\label{Larman-Rogers}
m_1(\R^n)\leq \frac{\alpha(G)}{|V|}.
\end{equation}
Combined with the celebrated Frankl and Wilson intersection theorem \cite{FW},
this inequality has lead to the  asymptotic upper bound of
$1.207^{-n}$,  proving  the exponential decrease of $m_1(\R^n)$. This
result was  later
improved to $1.239^{-n}$ in \cite{Rai} following similar
ideas. However, \eqref{Larman-Rogers} can by no means result in a lower
estimate for $\chi_m(\R^n)$ that would be tighter than that of
$\chi(\R^n)$ since the inequalities  $\chi(\R^n)\geq \chi(G)\geq
\frac{|V|}{\alpha(G)}$ obviously hold.
In \cite{SW}, a more sophisticated configuration principle was
introduced that improved the upper bounds of $m_1(\R^n)$ for
dimensions $2\leq n\leq 25$, but didn't move forward to an asymptotic improvement.

A completely different approach is taken in \cite{OV}, where an analogue of Lov\'asz theta number is
defined and computed for the unit distance graph (see also \cite{BNOV}
for an earlier approach dealing with the unit sphere of Euclidean
space). This number,  denoted here $\vartheta(\R^n)$, has an explicit
expression in terms of Bessel functions, which will be recalled in
section 2. However, the resulting upper bound of $m_1(\R^n)$ is asymptotically not as good as  Frankl and
Wilson. We introduce the  following notations, which will be used throughout this paper:

\smallskip

\noindent{\bf Notations:} Let $u_n$ and $v_n\neq 0$ be two sequences. We denote $u_n \sim v_n$ if
 $\lim u_n/v_n=1$, 
$u_n\approx v_n$ if there exists $\alpha, \beta \in \R$, $\beta > 0$,
such that 
$ u_n/v_n \sim \beta n^\alpha$ and, for positive
sequences,
$u_n\lessapprox v_n$ if there exists $\alpha ,\beta \in \R$, $\beta>0$
such that $ u_n/v_n\leq \beta n^\alpha$.

\smallskip
\noindent Then, the asymptotic behavior of $\vartheta(\R^n)$ is
\begin{equation*}
\vartheta(\R^n)\approx (\sqrt{e/2})^{-n} \lessapprox (1.165)^{-n}.
\end{equation*}
Nevertheless,  for small dimensions, $\vartheta(\R^n)$  does improve the previously
known upper bounds of $m_1(\R^n)$. 
Moreover, this bound  is further
strengthened in \cite{OV} by adding extra inequalities arising from
simplicial configurations of points, leading to the up to now tightest known
upper bounds of $m_1(\R^n)$ for $2\leq n\leq 24$ \cite[Table 3.1]{OV}. 

In this paper, we build upon the results in \cite{OV}, by considering
 more general configurations of
points. More precisely, a linear program is associated to any finite induced subgraph of the unit distance
graph $G=(V,E)$, whose optimal value
$\vartheta_G(\R^n)$ satisfies
\begin{equation*}
m_1(\R^n)\leq \vartheta_G(\R^n)\leq \vartheta(\R^n).
\end{equation*}
We prove that 
  $\vartheta_G(\R^n)$ decreases exponentially faster than both
$\vartheta(\R^n)$ and the ratio $\alpha(G)/|V|$, when $G$ is taken in the family
of graphs considered by Frankl and Wilson, or in the family of graphs
defined by Raigorodskii. We obtain the improved estimate
\begin{theorem}\label{Theorem 1}
\begin{equation}\label{final}
m_1(\R^n)\lessapprox (1.268)^{-n}.
\end{equation}
\end{theorem}

We also present numerical results in the range of dimensions $4\leq
n\leq 24$ (see Table \ref{Table2}),
where careful choices of graphs allow us to tighten the previously
known upper estimates of $m_1(\R^n)$.

For the first time, tightening a theta-like upper bound using a subgraph constraint is applied 
in a systematic way. Both our numerical results in small dimensions and the asymptotic improvement that we have 
obtained (even if further improvement with this method is intrinsically limited, see Remark \ref{rem-lim}) show the relevance of this method. We believe this is a promising method  that is worth 
to consider in other situations, in particular because it is much simpler than others such as further steps in hierarchies 
of semidefinite programs (see \cite{Las}, \cite{Lau}, and for packing graphs in topological spaces, \cite{LV}). In section 5, we show how it can be applied in the framework of compact homogeneous graphs.

\smallskip

This paper is organized as follows: in section \ref{section-LP}, we introduce 
$\vartheta_G(\R^n)$ and prove the announced inequality $m_1(\R^n)\leq \vartheta_G(\R^n)$. 
In section \ref{section-theorem 1}, we prove Theorem \ref{Theorem 1}. Section \ref{section-numericals}
is devoted to the numerical results in small dimension. Section \ref{section-theta} 
develops the method  in the  context of compact homogeneous graphs. The last section presents some open problems.

\section{A linear programming upper bound for $m_1(\R^n)$}\label{section-LP}

Our goal in this section is to generalize to arbitrary
induced subgraphs of the unit distance graph the linear programming upper bound of
$m_1(\R^n)$ introduced in \cite{OV}.  In order to formulate the result
we need some preparation.

Let $\omega$ denote the surface measure of the unit sphere
$S^{n-1}=\{x\in \R^n \ :\ \Vert x\Vert =1\}$ and let 
$\omega_n=\omega(S^{n-1})$. We will need  the Fourier transform of $\omega$, 
which is by definition the function defined for $u\in \R^n$ by $\widehat{\omega}(u)=\int_{S^{n-1}}e^{-iu\cdot
  \xi}d\omega(\xi)$. This function  is clearly
invariant under the orthogonal group $O(\R^n)$, so let  $\Omega_n(t)$ be the function of $t\geq 0$ such that, for all $u\in \R^n$,
\begin{equation*}
\Omega_n(\Vert u\Vert)=\frac{1}{\omega_n}\int_{S^{n-1}}e^{-iu\cdot \xi}
d\omega(\xi).
\end{equation*}
We note that $\Omega_n(0)=1$.
The function $\Omega_n$ expresses in terms of the Bessel function of
the first kind \cite[Chap.4 and Lemma 9.6.2]{AAR}, :
\begin{equation*}
\Omega_n(t) = 
\Gamma\left(\frac{n}{2}\right)
\left(\frac{2}{t}\right)^{\frac{n}{2}-1}
J_{\frac{n}{2}-1}(t).
\end{equation*}

In \cite[section 3.1]{OV}, the following theorem is proved:

\begin{theorem}\label{theorem OV} \cite{OV}
\begin{equation*}
\begin{array}{cl}
m_1(\R^n)\leq \inf \ \big\{ z_0+z_c \ :& z_c\geq 0, \\
& z_0+z_1+z_c(n+1)\geq 1\\
& z_0+z_1\Omega_n(t)+z_c(n+1)\Omega_n\big(t\sqrt{\frac{1}{2}-\frac{1}{2n+2}}\big)\geq 0\\
 & \text{ for all }t\geq 0\ \big\}
\end{array}
\end{equation*}
\end{theorem}

Let $G=(V,E)$ a (not necessarily finite) induced subgraph of the unit
distance graph. So, $V\subset \R^n$ and  $E=\{\{x,y\} \ : \ x\in V, y\in V \text{ and } \Vert x-y\Vert=1\}$. 
We assume that $V$ is a Borel measurable set, endowed with a positive
Borel measure $\lambda$, such that $0<\lambda(V)<+\infty$. We introduce
the {\em $\lambda$-independence number} of $G$:
\begin{equation*}
\alpha_{\lambda}(G):=\sup\{\lambda(A)\ :\ A\subset V,\ A\text{
a Borel  measurable independent set of }G \}.
\end{equation*}
If $G$ is a finite graph, and $\lambda$ is the counting measure, 
we recover the usual notion of the independence number $\alpha(G)$
of $G$.

\begin{theorem}\label{theorem-general}
With the notations above,
\begin{equation*}
\begin{array}{cl}
m_1(\R^n)\leq \inf \ \big\{ z_0+z_2\frac{\alpha_{\lambda}(G)}{\lambda(V)} \ :& z_2\geq 0, \\
& z_0+z_1+z_2\geq 1\\
& z_0+z_1\Omega_n(t)+z_2\frac{1}{\lambda(V)}\int_V\Omega_n(t\Vert
v\Vert )d\lambda(v)\geq 0\\
 & \text{ for all }t\geq 0\ \big\}
\end{array}
\end{equation*}
The optimal value of this linear program will be denoted $\vartheta_G(\R^n)$.
\end{theorem}

Before we proceed with the proof, we would like to make a few comments
on the choice of subgraph $G$.

\smallskip
In the next sections, we will apply Theorem \ref{theorem-general} in
the special case of a subgraph $G=(V,E)$ such that $V$ is finite and
lies on the sphere of radius $r$ centered at $0^n$, and the measure $\lambda$ is the
counting measure. Then, the linear program takes the simpler form 
\begin{equation}\label{LP}
\begin{array}{cl}
\inf \ \big\{ z_0+z_2\frac{\alpha(G)}{|V|} \ :& z_2\geq 0, \\
& z_0+z_1+z_2\geq 1\\
& z_0+z_1\Omega_n(t)+z_2\Omega_n(rt)\geq 0\text{ for all }t\geq 0\ \big\}.
\end{array}
\end{equation}
In particular, if $V$ is a regular simplex with edges of length $1$, centered at $0^n$, the graph $G$ is
the complete graph of order $(n+1)$;  the change of variable
$z_2=z_c(n+1)$ in \eqref{LP}, combined with Theorem
\ref{theorem-general},  gives back  Theorem \ref{theorem OV}.

Another natural choice is to take for $V$ a sphere centered at $0^n$, endowed with its surface
measure. Of course, the exact value of the
ratio $\frac{\alpha_{\lambda}(G)}{\lambda(V)}$ is not known in
general, but one can upper bound it with a similar linear
program. This will be  explained in section \ref{section-theta}, where the more general case of compact homogeneous graphs is
discussed. 

It is also possible to take account of several graphs 
at the same time; each graph would give rise to an additional variable $z_i$.
\smallskip

Theorem \ref{theorem-general} can be obtained from a minor
modification of the argument in \cite{OV}. In order to keep this paper
self contained, we include a  proof here. 

\begin{proof}
One can come arbitrary close to
$m_1(\R^n)$, with a Borel measurable subset of $\R^n$ avoiding distance $1$,
which is moreover a periodic set. We refer to \cite[proof of Theorem
1.1]{OV} for a proof. So, let $S$ be a periodic Borel measurable subset of
$\R^n$ avoiding distance $1$, having positive density, and let $L$ denote its periodicity lattice. We consider the function:
\begin{equation*}
f_S(x):=\frac{1}{\vol(L)}\int_{\R^n/L} \1_S(x+y)\1_S(y) dy
\end{equation*}
where $\1_S(x)$ denotes the characteristic function of $S$ and integration is with respect to the Lebesgue measure.
One can verify that $f_S(0^n)=\delta(S)$, that $f_S$ is $L$-periodic  and that $\frac{1}{\vol(L)}\int_{\R^n/L} f_S(x)dx=\delta(S)^2$.
Moreover, $f_S$ is a {\em positive definite function}, meaning that, for all choice of $k$ points in $\R^n$, say 
$x_1,\dots,x_k$, the matrix with coefficients $f_S(x_i-x_j)$ is
positive semidefinite (see \cite[1.4.1]{Ru}). 

$\delta_{0^n}$ denotes the Dirac measure
at $0^n$, and  $\tilde{\omega}$, $\tilde{\lambda}$ stand for the
natural extensions of $\omega$ and $\lambda$ to $\R^n$
(i.e. for $E\subset \R^n$ a Borel set, $\tilde{\omega}(E):=\omega(E\cap
S^{n-1})$ and $\tilde{\lambda}(E):=\lambda(E\cap V)$). Let
$(z_0,z_1,z_2)\in \R^3$ and let the Borel measure 
\begin{equation*}
\mu:=z_0\delta_{0^n}+z_1\tilde{\omega}/\omega_n+z_2\tilde{\lambda}/\lambda(V).
\end{equation*}
We assume that  $z_2\geq 0$, that $\widehat{\mu}(u)\geq 0$ for all $u\in \R^n$, and that $\widehat{\mu}(0^n)\geq 1$.
Then, we claim that the following inequalities hold:
\begin{equation}\label{ineq0}
\delta(S)^2\leq \int f_S(x) d\mu(x) \leq \Big(z_0+z_2 \frac{\alpha_{\lambda}(G)}{\lambda(V)}\Big)\delta(S).
\end{equation}

To show the right hand side inequality, we observe that   $\int f_S d\delta_{0^n}=f_S(0^n)=\delta(S)$,
and, because $S$ avoids $1$, that $\int f_S d\tilde{\omega} =0$. Less obvious is the inequality 
\begin{equation*}
\int f_S(x) d\tilde{\lambda}(x) \leq \alpha_{\lambda}(G)\delta(S).
\end{equation*}
It is easily obtained from a swap of integrals following  Fubini's
theorem and from the inequality
\begin{equation*}
\int \1_S(x+y)d\tilde{\lambda}(x)=\lambda((S-y)\cap V)\leq \alpha_{\lambda}(G).
\end{equation*}

The left hand side inequality in \eqref{ineq0} follows from basic results in
Fourier ana\-lysis for which we refer to \cite[Chapter 1]{Ru}. Because
$f_S$ is continuous, positive definite on $\R^n$ and $L$-periodic, 
\begin{equation*}
f_S(x)=\sum_{\gamma\in 2\pi L^*} \widehat{f_S}(\gamma)
e^{i(\gamma\cdot x)}
\end{equation*}
where $L^*=\{x\in \R^n \ :\ x\cdot y\in \Z
\text{ fo all } y\in L\}$ is the dual lattice of $L$, and its Fourier coefficients
\begin{equation*}
\widehat{f_S}(\gamma)=\frac{1}{\vol(L)}\int_{\R^n/L}
f_S(y)e^{-i(\gamma\cdot y)}dy
\end{equation*}
are non negative numbers 
(see Bochner's theorem \cite[1.4.3]{Ru} and the inversion theorem
\cite[1.5.1]{Ru}).
Then, applying Fubini, and the assumptions on $\widehat{\mu}$,
\begin{equation*}
\int f_S(x) d\mu(x) =\sum_{\gamma\in 2\pi L^*} \widehat{f_S}(\gamma)
\widehat{\mu}(-\gamma)\geq \widehat{f_S}(0^n)=\delta(S)^2.
\end{equation*}
So, from \eqref{ineq0},
\begin{equation}\label{ineq-lambda}
\delta(S)\leq z_0+z_2 \frac{\alpha_{\lambda}(G)}{\lambda(V)}.
\end{equation}
Now, we introduce the measure $\lambda_0$, which is the average of
$\lambda$ with respect to the normalized Haar
measure $dg$  on  the orthogonal group $O(\R^n)$, and we will apply
\eqref{ineq-lambda} to $\lambda_0$.
The measure $\lambda_0$ is
defined by: $\lambda_0(E)=\int_{O(\R^n)} \tilde{\lambda}(g^{-1}(E))dg$
for any Borel set $E\subset \R^n$ and we
note that its support may be different than that of $\lambda$,
but is anyway contained in the union $V_0$ of the images of $V$ under
elements of $O(\R^n)$. We have $\lambda_0(V_0)=\lambda(V)$ and, with obvious notations, 
$\alpha_{\lambda_0}(G_0)\leq \alpha_{\lambda}(G)$. The Fourier transform of $\lambda_0$
can be written:
\begin{equation*}
\widehat{\lambda_0}(u)=\int \Omega_n(\Vert u \Vert \, \Vert
x\Vert)d\tilde{\lambda}(x) =\int_V \Omega_n(\Vert u \Vert \, \Vert
v\Vert)d\lambda(v)
\end{equation*}
so, if $\mu_0:=z_0\delta_{0^n}+z_1
\tilde{\omega}/\omega_n+z_2\lambda_0/\lambda_0(V_0)$, the conditions
that $\widehat{\mu_0}(u)\geq 0$ for all $u\in \R^n$ and
$\widehat{\mu_0}(0^n)\geq 1$ translate to:
\begin{equation*}
\begin{cases}
z_0+z_1\Omega_n(\Vert u\Vert )+z_2 \frac{1}{\lambda(V)}\int_V \Omega_n(\Vert u \Vert \, \Vert
v\Vert)d\lambda(v)\geq 0 \text{ for all }u\in \R^n\\
z_0+z_1+z_2\geq 1
\end{cases}
\end{equation*}
which amounts, together with $z_2\geq 0$,  to  $(z_0,z_1,z_2)$ being feasible for the linear
program in the theorem. Under these conditions, \eqref{ineq-lambda}
holds for
$\lambda_0$ and concludes the proof.
\end{proof}

We  recall from \cite[section 3]{OV} and \cite{Oli} that the linear program
obtained from $\vartheta_G(\R^n)$ when the variable $z_2$ is set to
$0$, can be solved in full generality and that its optimal value,
denoted here $\vartheta(\R^n)$,
has the explicit expression:
\begin{equation}\label{theta}
\vartheta(\R^n)=\frac{-\Omega_n(j_{n/2,1})}{1-\Omega_n(j_{n/2,1})}
\end{equation}
where $j_{n/2,1}$ is the first positive zero of $J_{n/2}$ and is the
value at which the function $\Omega_n$ reaches its absolute minimum
(see Figure 1 for a plot of $\Omega_4(t)$).
\begin{figure}
\centering
\includegraphics[width=0.7\textwidth]{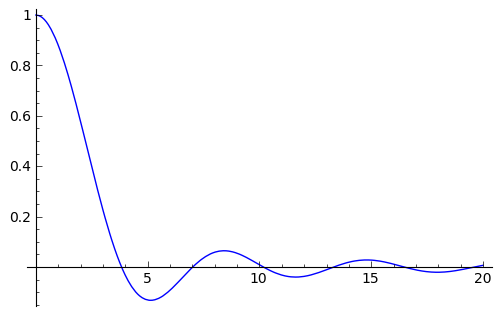}
\caption{$\Omega_4(t)$}
\end{figure}
In particular, we have the inequalities:

\begin{equation*}
m_1(\R^n)\leq \vartheta_G(\R^n)\leq \vartheta(\R^n).
\end{equation*}

Unfortunately, the program $\vartheta_G(\R^n)$ 
cannot be solved explicitly in a similar fashion. Instead, we will
content ourselves with the construction of explicit feasible solutions
in section \ref{section-theorem 1} and with numerical solutions in
section \ref{section-numericals}.

\section{The proof of Theorem 1}\label{section-theorem 1}

In this section we will show that from Theorem \ref{theorem-general} an asymptotic
improvement of the known upper bounds for $m_1(\R^n)$ can be
obtained. For this, we assume that $G_n=(V_n,E_n)$ is  a sequence of
induced subgraphs 
of the unit distance graph of dimension $n$, such that $|V_n|=M_n$ and 
$V_n$ lies on the sphere of radius $r<1$ centered at $0^n$, where $r$ is independent of
$n$. We recall from Theorem \ref{theorem-general} and \eqref{LP} that $m_1(\R^n)\leq
\vartheta_{G_n}(\R^n)$  where $\vartheta_{G_n}(\R^n)$ is the optimal
value of:
\begin{equation}\label{thetaG}
\begin{array}{cl}
\inf\big\{ z_0+z_2\frac{\alpha(G_n)}{M_n} :\
 & z_2\geq 0\\
& z_0+z_1+z_2 \geq 1\\
& z_0+z_1\Omega_n(t)+z_2 \Omega_n(rt) \geq 0 \  (t> 0)\big\}.
\end{array}
\end{equation}
So, in order to upper bound $m_1(\R^n)$, it is enough to construct a
 feasible solution of \eqref{thetaG}. One that is suitable for our
 purpose is given in the following lemma:

\begin{lemma}\label{lemma} For $0<r<1$, let $c(r)$ be defined by: 
\begin{equation*}
c(r) :=(1+\sqrt{1-r^2})\e^{-\sqrt{1-r^2}}
\end{equation*} 
(see the plot of this function in Figure 2). 

Let
$\gamma > \sqrt{c(r)}$ and $m > \gamma \sqrt{2/\e}$;  there
exists $n_0\in \N$ such that for all $n\geq n_0$, 
\begin{equation}\label{ineq}
m^n+ \Omega_n(t)+\gamma^n\Omega_n(rt)\geq 0, 
\text{ for all } t \geq 0.
\end{equation}
\end{lemma}

\begin{figure}
 \centering
 \includegraphics[width=0.7\textwidth]{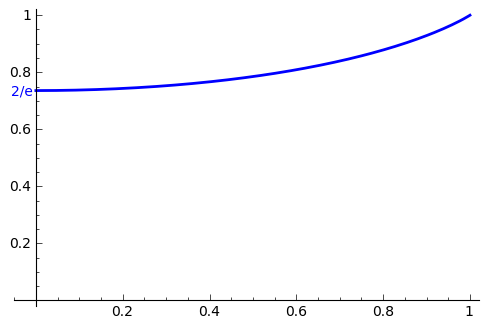}
 \caption{$c(r)$}
 \end{figure}

\begin{proof} After having established some preliminary inequalities,
  we will proceed in three steps. First, we will prove that the
  inequality \eqref{ineq} holds for ``small'' $t$, say 
$0\leq t \leq \nu := \frac{n}{2}-1$,
  then that it holds for ``large'' $t$, say $t \geq \alpha_0\nu$
  where $\alpha_0$ is an explicit constant, and, at last, we will construct
  a decreasing sequence $\alpha_0\geq \alpha_1\geq \ldots \geq
  \alpha_k\ldots$ such that the inequality holds for $t\geq \alpha_k\nu$,
and prove that $\lim_{k\rightarrow \infty}
 \alpha _k < 1$.

Let $j_{n/2,1}$ be the first zero of 
$J_{\nu+1}$, then $\Omega_n$ is a decreasing function  on 
$[0,j_{n/2,1}]$ and $\Omega_n$ has a  global minimum at  $j_{n/2,1}$ (it follows from \cite[(4.6.2), (4.14.1)]{AAR} and \cite[15.31]{Wat},
a detailed proof is given in \cite[sec. 4.3, (4.17)]{Oli}).
So,  $ \Omega_n(t)
\geq  \Omega_n(j_{n/2,1})$. Furthermore, $\vert J_\nu(t)\vert \leq 1$ 
for all $t\in \R$
(see \cite[formula 9.1.60]{AS}), hence
$$
\vert  \Omega_n(j_{n/2,1})\vert \leq \Gamma\left(\frac{n}{2}\right)
\left(\frac{2}{j_{n/2,1}}\right)^{\frac{n}{2}-1}.
$$
We apply the inequality $j_{n/2,1}> n/2$ (\cite[15.3 (1)]{Wat})
and  Stirling formula, which with our notation reads:
\begin{equation}\label{eq stirling}
\Gamma(n)\approx \left(\frac{n}{\e}\right)^n,
\end{equation}
and we get
\begin{equation}\label{eq11}
\vert  \Omega_n(j_{n/2,1})\vert \lessapprox \left(\sqrt{\frac{2}{\e}}\right)^n.
\end{equation}

Let $x\in ]0,1[$ and let us recall \cite[formula 9.3.2]{AS} (see also \cite[8.4 (3)]{Wat}):
\begin{equation}\label{eq11.0}
J_{\nu}(\nu \sech \alpha) \sim \frac{\e^{\nu(\tanh \alpha-\alpha)}}{\sqrt{2\pi\nu\tanh \alpha}} \qquad (\alpha >0,\ \nu\to +\infty).
\end{equation}
Setting $x=\sech \alpha$, \eqref{eq11.0} with our notation, leads to:
\begin{equation}\label{eq11.1}
J_{\nu}(x\nu) \approx \left(\sqrt{\frac{x}{c(x)}}\right)^n.
\end{equation}
We note that \eqref{eq11.0} shows that, for $n$ sufficiently large (possibly depending on $x$), $J_{\nu}(x\nu)$, and thus $\Omega_n(x\nu)$ is positive.

Combining \eqref{eq11.1} and \eqref{eq stirling}, we get
\begin{equation} \label{eq1}
\Omega_n(x\nu) = \Gamma\left(\frac{n}{2}\right)
\left(\frac{2}{x\nu}\right)^{\nu}
J_{\nu}(x\nu) \approx \left(\sqrt{\frac{2}{\e c(x)}}\right)^n.
\end{equation}

\noindent{\sl First step.} Suppose that $0\leq t \leq \nu$. Since
$\nu \leq j_{n/2,1}$ and $r < 1$, and since $\Omega_n$ is decreasing on $[0,j_{n/2,1}]$, it follows that 
$\Omega_n(rt) \geq \Omega_n(r\nu)$, so we have the inequality:
$$\Omega_n(t)+\gamma^n\Omega_n(rt) \geq -\vert
\Omega_n(j_{n/2,1})\vert+\gamma^n\Omega_n(r\nu).$$
We assume that $\Omega_n(r\mu)\geq 0$, which holds for $n$ sufficiently large as a consequence of \eqref{eq11.0}.
 From \eqref{eq11}, \eqref{eq1}, and the assumption $\gamma>\sqrt{c(r)}$,  the second term of the right hand side has 
asymptotically the largest absolute value, so, for $n$ greater  than
some value $m_0$, the sign of 
the right hand side is that of $\Omega_n(r\nu)$, hence is positive. So, for $n\geq m_0$, and for all $t\in [0,\nu]$, we have
\begin{equation*}
\Omega_n(t)+\gamma^n\Omega_n(rt)\geq 0.
\end{equation*}

\noindent{\sl Second step.} Let $\alpha_0=\frac{1}{\gamma^2}$. 
For $t\geq \alpha_0\nu$, because $|J_{\nu}(t)|\leq 1$,
$$\vert\Omega_n(t)\vert \leq  \Gamma\left(\frac{n}{2}\right)
\left(\frac{2}{\alpha_0\nu}\right)^{\frac{n}{2}-1}
\approx 
\left(\gamma\sqrt{\frac{2}{\e}}\right)^n.
$$
Since $\Omega_n(rt) \geq -\vert  \Omega_n(j_{n/2,1})\vert $ and $\vert
\Omega_n(j_{n/2,1})\vert \lessapprox \left(\sqrt{\frac{2}{\e}}\right)^n$, it
follows from the assumption $m>\gamma\sqrt{2/\e}$ that
$$m^n+\Omega_n(t)+\gamma^n\Omega_n(rt) \geq m^n -\vert\Omega_n(t)\vert
-\gamma^n\vert  \Omega_n(j_{n/2,1})\vert \sim m^n$$
and hence that, for $n$ greater that some $m_1$, and for all $t\geq \alpha_0\nu$,
\begin{equation*}
m^n+\Omega_n(t)+\gamma^n\Omega_n(rt) \geq 0.
\end{equation*}

\noindent{\sl Third step.} Let us first study the function $c$. An elementary 
computation gives $c'(x) = x\e^{-\sqrt{1-x^2}}$ 
for $x\in
[0,1]$. It implies that $0\leq c'(x) \leq 1$, hence $c$ is an
increasing function and $c(x) \geq x$ with equality only for $x =
1$. Now, let us define $\phi$ by
$$
\begin{array}{rcl}
\phi : [0,\frac{1}{r}] & \longrightarrow & \R\\
x & \mapsto & \frac12(\frac{c(rx)}{\gamma^2}+x)
\end{array}
$$
Since $c$ is increasing, $\phi$ is also increasing. Furthermore,
$\phi(0) = \frac1{\e\gamma^2} > 0$ and $\phi(\frac{1}{r}) = 
\frac12(\frac{1}{\gamma^2}+\frac{1}{r}) <\frac{1}{r}$ since $\gamma^2
> c(r) > r$. It follows that the interval $[0,\frac{1}{r}]$ is mapped
into itself. One also gets immediately $\phi'(x) = 
\frac12(\frac{rc'(rx)}{\gamma^2}+1)$. Since $c'(rx) \leq 1$ and
$\gamma^2> r$, we have $\phi'(x) < 1$. Hence by Banach fixed point theorem, $\phi$ has only one fixed point,
denoted by $l$ and, for any
$x_0\geq l$, the sequence $\{x_k,\, k\geq 0\}$ defined by $x_{k+1} = \phi(x_k)$ is a decreasing
sequence with limit $l$. Moreover, $\phi(1)<1$, so $l<1$.

We now return to the proof of the lemma. We have set
$\alpha_0=\frac{1}{\gamma^2}$. If $\alpha_0\leq 1$, taking account of the previous steps, we are done, so, we assume $\alpha_0>1$. 
Let $\alpha_1 < \alpha_0$ and
$t\in [\alpha_1\nu,\alpha_0\nu]$. By construction, $ r\alpha_0=
\frac{r}{\gamma^2}< \frac{r}{c(r)} < 1$, hence $rt \leq r\alpha_0\nu
<\nu \leq j_{n/2,1}$. Since $\Omega_n$ is decreasing on $[0,j_{n/2,1}]$, \eqref{eq1} gives
$$\Omega_n(rt) \geq \Omega_n(r\alpha_0\nu) \approx 
\left(\sqrt{\frac{2}{\e c(r\alpha_0)}}\right)^n.$$
Now $\vert\Omega_n(t)\vert \leq  \Gamma\left(\frac{n}{2}\right)
\left(\frac{2}{\alpha_1\nu}\right)^{\frac{n}{2}-1}
\approx 
\left(\sqrt{\frac{2}{\e\alpha_1}}\right)^n$. Hence, with the same reasoning as in step 1, we will have 
$\Omega_n(t)+\gamma^n\Omega_n(rt) \geq 0$, for $n$ greater than some value $m_2$,
 as soon as $\alpha_1 > \frac{c(r\alpha_0)}{\gamma^2}$ (here we need
 strict inequality). In order to achieve this constraint, we can take
 $\alpha_1= \phi(\alpha_0)$. Defining the sequence $\{\alpha_k, k\geq 0\}$ by the recursive formula
$\alpha_{k+1} = \phi(\alpha_k)$, we get, using the same method, that
for every $k\geq 1$, there exists $m_{k+1}$, such that for all $n>m_{k+1}$,
$\Omega_n(t)+\gamma^n\Omega_n(rt) \geq 0$ for $t \geq
\alpha_k\nu$. Since $\lim \alpha_k= l < 1$, there exists an integer  $k_0$ such that
$\alpha_{k_0}< 1$. 

\noindent{\sl Conclusion.} With the three steps above, we have covered
the whole range $t\geq 0$ by a finite number of intervals, and proved
the wanted result on each of them. All together, we have that, 
for $n\geq n_0:=\max(m_0,m_1,\dots,m_{k_0+1})$, 
$m^n+\Omega_n(t)+\gamma^n\omega_n(rt)\geq 0$ for all $t\geq 0$.
\end{proof}

\begin{theorem}\label{Theorem 2}
We assume that, for some $b<\sqrt{2/\e}$,
\begin{equation}\label{bound b}
\frac{\alpha(G_n)}{M_n}\lessapprox b^n.
\end{equation}
Let 
\begin{equation*}
c(r) =(1+\sqrt{1-r^2})\e^{-\sqrt{1-r^2}}\quad \text{ and} \quad f(r)=\sqrt{2c(r)/\e}.
\end{equation*}
Then, for every $\epsilon>0$, 
\begin{equation}\label{bound thetaG}
\vartheta_{G_n}(\R^n)\lessapprox (f(r)+\epsilon)^n.
\end{equation}
\end{theorem}

\begin{proof} Let $\epsilon>0$; let
$\gamma=\sqrt{c(r)}+\epsilon$ and $m=\sqrt{2c(r)/\e}+\epsilon$.
Lemma \ref{lemma} shows that for $n$ sufficiently large,
$(z_0,z_1,z_2)=(m^n, 1,\gamma^n)$ is a feasible solution of
\eqref{thetaG}. So, for these values of $n$, the optimal value of
\eqref{thetaG} is upper bounded by $m^n+\gamma^n \alpha(G_n)/M_n$,
leading to
\begin{equation*}
\vartheta_{G_n}(\R^n)\lessapprox \big(\sqrt{2c(r)/\e}+\epsilon\big)^n+\big(b\sqrt{c(r)}+\epsilon\big)^n
\lessapprox (f(r)+\epsilon)^n.
\end{equation*}

\end{proof}

In order to complete the proof of Theorem
\ref{Theorem 1}, we will apply Theorem \ref{Theorem 2} to a certain sequence of graphs introduced 
by Raigorodskii in \cite{Rai}. Before that, we introduce the family of {\em generalized Johnson
  graphs} and recall Frankl and Wilson upper bound on their independence number. We will explain with full details how this bound, combined with Theorem \ref{Theorem 2}, reaches the inequality
\begin{equation*}
m_1(\R^n)\lessapprox (1.262)^{-n}.
\end{equation*}
Then, we will proceed to the graphs considered in \cite{Rai}, which
will allow us to obtain the slightly better bound announced in Theorem
\ref{Theorem 1} with similar techniques.

\begin{definition}\label{def Johnson}
We denote $J(n,w,i)$ and call \emph{generalized Johnson graph}  the graph with vertices the set of
$n$-tuples of $0$'s and $1$'s, with $w$ coordinates equal to
$1$, and with edges connecting pairs of $n$-tuples having exactly $i$
coordinates in common equal to $1$. 
\end{definition}

An upper bound of
$\alpha(J(n,w,i))$ is provided
by Frankl and Wilson intersection
theorem \cite{FW} and applies for certain values of the parameters $w$ and $i$:

\begin{theorem}\label{Theorem FW}\cite{FW}
If $q$ is a power of a prime number, 
\begin{equation*}
\alpha(J(n,2q-1,q-1))\leq \binom{n}{q-1}.
\end{equation*}
\end{theorem}

Standard results on the density of prime numbers ensure that, for all $a>0$, there exists a sequence of primes that grow like $an$.
Indeed, the prime number theorem states
that  the $m$-th prime number  $\pi(m)$ satisfies $\pi(m) \sim
m\ln(m)$ where $\ln(t)$ denotes  the natural base-$e$ logarithm function  (see \cite{arias} for a historical survey). 
Since $m\ln(m)$ is strictly increasing 
to infinity, for any $n\in \N$ there exists $m_n \geq 1$ such that 
\begin{equation*}
m_n\ln(m_n) \leq an < (m_n+1)\ln(m_n+1).
\end{equation*}
In particular $an \sim m_n\ln(m_n)$.
We set $p_n = \pi(m_n)$, and then $p_n \sim m_n\ln(m_n) \sim an$.

Taking $q=p_n\sim an$, with $a<1/4$ and $H(a):=-a\ln(a)-(1-a)\ln(1-a)$, Theorem \ref{Theorem FW} leads to 
\begin{equation}\label{bound FW}
\frac{\alpha(J(n,2q-1,q-1))}{|J(n,2q-1,q-1)|}\leq
\frac{\binom{n}{q-1}}{\binom{n}{2q-1}}\approx \e^{-(H(2a)-H(a))n}.
\end{equation}
Moreover, 
this upper bound is optimal in the sense that it cannot be tighten by an exponential factor
(see \cite{AK}). The optimal choice of $a$, i.e. the value of $a$  that maximizes $H(2a)-H(a)$ is  $a=(2-\sqrt{2})/4$, from which one obtains the
upper estimate $(1.207)^{-n}$. Let us recall that this result gave
the first  lower estimate of exponential growth for the chromatic
number of $\R^n$ \cite{FW}.

Let  $G_n=J(n,2p_n-1,p_n-1)$ where $p_n$ is, as above,  a sequence
of prime numbers such that $p_n \sim an$. The value of  $a<1/4$ will be chosen later
in order to optimize the resulting bound \eqref{bound thetaG}
(interestingly, it will turn to be different than the one that optimizes \eqref{bound FW}).
So, we have $M_n=\binom{n}{2p_n-1}$ and the constant $b$ in \eqref{bound b} is $b(a)=\e^{-(H(2a)-H(a))}$.

From $G_n$, we construct  unit distance graphs  in $\R^n$ by assigning 
the real value $t_0$ to the $0$-coordinates and the real value $t_1$ to the $1$-coordinates.
The squared Euclidean distance between two vertices is 
equal to $2(t_0-t_1)^2p_n$ so, assuming $t_0\geq t_1$, we must have
$t_0=t_1+1/\sqrt{2p_n}$. The vertices then belong to the sphere
centered at $0^n$ and of radius $r$ with $r^2=t_0^2(n-2p_n+1)+t_1^2w=t_1^2n+2t_1(n-2p_n+1)/\sqrt{2p_n}+(n-2p_n+1)/2p_n$.
So, we obtain infinitely many induced subgraphs of the unit distance
graph, all of them being combinatorially equivalent to $G_n$, and realizing every
radius $r$ such that 
\begin{equation*}
r\geq
r_{\min}(n,p_n):=\sqrt{ \frac{(n-2p_n+1)(2p_n-1)}{2n p_n }}
\sim
r(a):=\sqrt{1-2a}.
\end{equation*} 

The function
$f(r(a))=\sqrt{2c(r(a))/\e}$ is decreasing with $a$, so we will take
the largest possible value for $a$, under the constraint $b(a)\leq
\sqrt{2/\e}$. Let this value be denoted $a_0$; then 
$b(a_0)=\sqrt{2/\e}$ and  (see
Figure 3)
\begin{equation*}
0.2268\leq a_0 \leq 0.2269.
\end{equation*}

\begin{figure}
 \centering
 \includegraphics[width=0.7\textwidth]{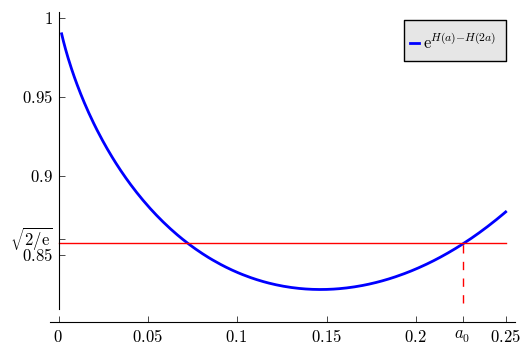}
 \caption{$\e^{H(a)-H(2a)}$}
 \end{figure}

 We fix now $a=a_0$. For a given $\epsilon>0$, because the function
  $ f(r)=\sqrt{2c(r)/\e}$ is continuous, there is a $r>r(a_0)$ such
  that $f(r)=f(r(a_0))+\epsilon$,
  and such that,  for $n$ sufficiently large, 
$r$ is a valid radius for all the graphs $G_n$. Applying
   Theorem \ref{Theorem 2} to this value of $r$ and to $\epsilon$, we
   obtain 
\begin{equation*}
\vartheta_{G_n}(\R^n) \lessapprox (f(r(a_0))+\epsilon)^n
\end{equation*}
with 
\begin{equation*}
f(r(a_0))=\sqrt{2(1+\sqrt{2a_0})\e^{-(1+\sqrt{2a_0})}} <  (1.262)^{-1}.
\end{equation*}

In \cite{Rai}, Raigorodskii considers graphs with vertices in
$\{-1,0,1\}^n$, where the number of $-1$, respectively of $1$, 
grows linearly with $n$. If the number of $1$, respectively of $-1$,
is of the order of $x_1n$, respectively  $x_2n$, with $x_2\leq x_1$, if
$z=(x_1+3x_2)/2$, and $y_1=(-1+\sqrt{-3z^2+6z+1})/3$, he shows that:
\begin{equation}\label{Raigo}
\frac{\alpha(G_n)}{M_n}\lessapprox b(x_1,x_2)^n \text{ where }
b(x_1,x_2)=\e^{-(H_2(x_1,x_2)-H_2(y_1,(z-y_1/2))}
\end{equation}
where $H_2(u,v)=-u\ln(u)-v\ln(v)-(1-u-v)\ln(1-u-v)$. The proof of
\eqref{Raigo} relies on a similar argument as in Frankl-Wilson
intersection theorem.
These graphs can be realized as subgraphs of the unit distance graph
in $\R^n$ with minimal radius 
\begin{equation*}
r(x_1,x_2)=\sqrt{\frac{(x_1+x_2)-(x_1-x_2)^2}{(x_1+3x_2)}}.
\end{equation*}

For $x_1=0.22$ and $x_2=0.20$, the inequality $b(x_1,x_2)<\sqrt{2/\e}$
holds and $f(r(x_1,x_2))<1.268^{-1}$, leading to the announced
inequality \eqref{final}.

\begin{remark}\label{rem-lim} The possibility to further improve the basis of
  exponential growth using Theorem \ref{Theorem 2} is rather
  limited. Indeed, $f(r)\geq \sqrt{2c(1/2)/\e}>(1.316)^{-1}$. 
So, with this method, we cannot reach a better basis that $1.316$.
\end{remark}

\section{Numerical results for dimensions up to
  $24$}\label{section-numericals}

In this range of dimensions, we have tried many graphs in order to
improve the known upper estimates of $m_1(\R^n)$ (and, in turn, the
lower estimates of the measurable chromatic number). We didn't improve upon the results obtained in \cite{OV}
for dimension $2$ and $3$. We report here
the best we could achieve for $4\leq n\leq 24$. Table \ref{Table2} displays a feasible
solution $(z_0,z_1,z_2)$ of \eqref{LP} where the notations are
those of section \ref{section-LP}: $G$ is an induced subgraph of
the unit distance graph in dimension $n$, and it has $M$ vertices at
distance $r$ from $0^n$. The number given in the third column is the
exact value, or an upper bound, of its independence number
$\alpha(G)$, and replaces $\alpha(G)$ in \eqref{LP}. The last
column contains the objective value of \eqref{LP}, thus an upper
bound for $m_1(\R^n)$. Table \ref{Table3} gives the corresponding
lower bounds for $\chi_m(\R^n)$, compared to the previous best known
ones.

The computation of $(z_0,z_1,z_2)$ was
performed in a similar way as in \cite{OV}: the interval $[0,50]$ is
sampled in order  to replace the condition $ z_0+z_1\Omega_n(t)+z_2
\Omega_n(rt) \geq 0$ for all $t>0$ by a finite number of inequalities;
the resulting linear program is solved leading to a solution
$(z_0^*,z_1^*,z_2^*)$. The function $ z_0^*+z_1^*\Omega_n(t)+z_2^*
\Omega_n(rt)$ is then almost feasible for \eqref{LP}, in the sense
that its absolute minimum is reached in the range $[0,50]$ and is a
(small) negative number. Then, we need only
slightly increase $z_0^*$  in order to turn it to  a true feasible
solution.  The computations were performed with the help of the
softwares SAGE \cite{Sage} and lpsolve \cite{BEN}. 

In the next two subsections, we give more details on the graphs involved in the computations and
on how we dealt with their independence number.

\subsection{Johnson graphs}

The generalized Johnson graphs were introduced  in section \ref{section-theorem 1}.
They give the best upper bound for dimensions between $12$ and $23$.

According to the definition \ref{def Johnson},  the coordinates of any vertex of $J(n,w,i)$
sum to $w$, and the squared Euclidean distance
between two vertices connected by an edge is equal to $2(w-i)$, so,
after rescaling by $1/\sqrt{2(w-i)}$,
$J(n,w,i)$ is an induced subgraph of the unit distance graph of dimension
$n-1$. A straightforward calculation shows that it lies on a sphere
of radius equal to $\sqrt{w(1-w/n)/(2(w-i))}$.

For these graphs, several strategies are available to deal with their independence number, that we will discuss now.

A direct computation of $\alpha(J(n,w,i))$ for all $w,i$ turns successful  only up to  $n=10$ (we
have performed the computations using the package GRAPE of the
computational system  GAP, that deals with graphs with symmetries
\cite{Soi}; indeed, the graphs $J(n,w,i)$ are invariant under the
group of permutations of the $n$ coordinates). 
On the other hand, for the graphs $J(n,3,1)$, there is an
explicit formula due to Erd\H{o}s and S\'os (see \cite[Lemma
18]{LR}), but the number of vertices in this case, which is roughly equal to $n^3$,
it too small to lead to a good bound. 

If, being less demanding, we seek only for an upper estimate of
$\alpha(J(n,w,i))$, we have two possibilities. One of them is offered 
by Frankl and Wilson bound recalled in Theorem \ref{Theorem FW}, if the parameters $(n,w,i)$ 
are of the specific form $(n,2q-1,q-1)$ where $q$ is the power of a
prime number.

Another upper bound of  $\alpha(J(n,w,i))$ is given by the
Lov\'asz theta number $\vartheta(J(n,w,i))$ of the graph $J(n,w,i)$. The theta number of a graph was introduced by Lov\'asz in \cite{Lov}. 
It is a semidefinite programming relaxation of the independence number
(its definition and properties will be recalled in  section
\ref{section-theta}).
The group
of permutations of the $n$ coordinates acts transitively on the 
vertices as well as on the edges of the graph $J(n,w,i)$ so from \cite[Theorem 9]{Lov}, its
theta number expresses in terms of the largest and smallest
eigenvalues of the graph; taking into account that these eigenvalues,
being the eigenvalues of the Johnson scheme, are computed in
\cite{Del2} in terms of Hahn polynomials, we have, if $Z_k(i):=Q_k(w-i)/Q_k(0)$  with the notations
of \cite{Del2}:
\begin{equation}\label{theta-J}
\frac{\vartheta(J(n,w,i))}{|J(n,w,i)|} =\frac{-\min_{k\in [w]}
  Z_k(i)}{1-\min_{k\in [w]} Z_k(i)}.
\end{equation}

\begin{remark}
We note that this expression is completely analogous to \eqref{theta};  indeed,
both graphs afford an automorphism group that is edge transitive. We refer to
\cite{BDOV} for an interpretation of \eqref{theta} in terms of eigenvalues
of operators. 
\end{remark}

The bound on $\alpha(J(n,w,i))$ given by \eqref{theta-J}
unfortunately turns to be poor.  Computing $\vartheta'(J(n,w,i))$ instead of
$\vartheta(J(n,w,i))$ ($\vartheta'$ is a standard strengthening of $\vartheta$ obtained by adding a non negativity constraint on the matrix variables, see  Remark \ref{rem}) represents an easy way to
tighten it. Indeed, one can see that
$\vartheta(J(n,w,i))=\vartheta'(J(n,w,i))$ only if
\begin{equation}\label{cond}
Z_{k_0}(i)=\min_{j\in [w]} Z_{k_0}(j)
\end{equation}
where $k_0$ satisfies $Z_{k_0}(i)=\min_{k\in [w]}
Z_k(i)$. It turns out that \eqref{cond} is not always fulfilled
 and in these cases $\vartheta'(J(n,w,i))<\vartheta(J(n,w,i))$.

A more serious improvement is
provided by semidefinite programming following \cite[(67)]{Sch2} where
constant weight codes with given minimal distance are considered. 
In order to apply this framework to our setting, we only need to
change the range of avoided Hamming distances in \cite[(65-iv)]{Sch2}.

Table \ref{Table1} displays the numerical values of the three  bounds for certain
parameters $(n,w,i)$, selected either because they allow for Frankl and Wilson bound, or because they 
give the best upper bound for $m_1(\R^n)$ that we could achieve.  For most of these parameters, the semidefinite programming  bound turns to be the best
one and is significantly better than  the theta number. It would be of course very interesting to
understand the asymptotic behavior of this bound when $n$ grows to
$+\infty$, unfortunately this problem seems to be out of reach to
date.

The computation of the semidefinite programming bound  was performed
on the NEOS website (\url{http://www.neos-server.org/neos/}) with  the
solver SDPA \cite{SDPA} and double checked with SDPT3  version
4.0-beta \cite{TTT}.

\subsection{Other graphs}

The $600$-cell is a  regular polytope of dimension $4$ with
$120$ vertices: the sixteen
points $(\pm 1/2,\pm1/2, \pm 1/2,\pm 1/2)$, the eight permutations of $(\pm 1,0,0,0)$ and the 
$96$ points that are  even permutations of $(0,\pm 1/(2\phi), \pm 1/2,\pm \phi/2)$,  where $\phi=(1+\sqrt{5})/2$.
If $d$ is the distance between two non antipodal vertices, we have 
$d^2\in \{(5\pm\sqrt{5})/2, 3, (3\pm \sqrt{5})/2,  2,1\}$.
Each value of $d$ gives raise to a graph connecting the vertices that are at distance $d$ apart; these graphs,
after rescaling so that the edges have length $1$, lie on the sphere of radius $r=1/d$.
Their independence numbers are respectively equal to: $39, 26, 24, 26, 20$. We note that applying the conjugation $\sqrt{5}\to -\sqrt{5}$ will obviously not change the independence number.
Among these graphs, the best result in dimension $4$, recorded in Table \ref{Table2}, was obtained with $d=\sqrt{3}$.
 It turned out that the same graph gave the best result we could achieve in dimensions $5$ and $6$.

The root system $E_8$ is the following set of $240$ points in $\R^8$:
the points $((\pm 1)^2, 0^6)$ and all their permutations, and 
the points $((\pm 1/2)^8)$ with an even number of minus signs.
The distances between two non antipodal points take three different values: $d=\sqrt{2}, 2, \sqrt{6}$. 
The unit distance subgraph associated to a value of $d$ lies on a sphere of radius $r=\sqrt{2}/d$ and has an independence number 
equal respectively to  $16,16,36$. The one with the smallest radius $r=\sqrt{1/3}$ gives the best bound in dimension $8$
as well as in dimensions $9$, $10$, and $11$ (in these dimensions we have compared with Johnson graphs). 

The configuration in dimension $7$ is derived from $E_8$: given $p\in E_8$,  we take the set of points in $E_8$ closest to $p$. Independently of $p$, this construction leads to $56$ points that lie on a hyperplane. The graph  defined by the distance $\sqrt{6}$ after suitable rescaling corresponds to $r=\sqrt{6}/4$ and has independence number $7$.

In dimension $24$, we obtained the best result with the so-called orthogonality graph $\Omega(24)$. For $n=0\bmod 4$, $\Omega(n)$ denotes 
the graph with vertices in $\{0,1\}^{n}$, where the edges connect the points at Hamming distance $n/2$. Using semidefinite programming, an upper bound of its independence number is computed for $n=16,20,24$ in \cite{KP}.

\begin{table}
\centering
\bigskip
\bigskip

\begin{tabular}{|r|r|r|r|r|}
$(n,w,i)$ & $\alpha(J(n,w,i))$ & FW bound \cite{FW} & $\vartheta'(J(n,w,i))$  & SDP bound \\
\hline

$(6,3,1)$ &4&6&4&4\\
$(7,3,1)$ &5&7&5&5\\
$(8,3,1)$ &8&8&8&8\\
$(9,3,1)$ &8&9&11&8\\
$(10,5,2)$ &27&45&30&27\\
$(11,5,2)$ &37&55&42&37\\
$(12,5,2)$ &57&66&72&57\\
$(12,6,2)$ &&&130&112\\
$(13,5,2)$ &&78&109&72\\
$(13,6,2)$ &&&191&148\\
$(14,7,3)$ &&364&290&184\\
$(15,7,3)$ &&455&429&261\\
$(16,7,3)$ &&560&762&464\\
$(16,8,3)$ &&&1315&850\\
$(17,7,3)$ &&680&1215&570\\
$(17,8,3)$ &&&2002&1090\\
$(18,9,4)$ &&3060&3146&1460\\
$(19,9,4)$ &&3876&4862&2127\\
$(20,9,3)$ &&&13765&6708\\
$(20,9,4)$ &&4845&8840&3625\\
$(21,9,4)$ &&5985&14578&4875\\
$(21,10,4)$ &&&22794&8639\\
$(22,9,4)$ &&7315&22333&6480\\
$(22,11,5)$ &&&36791&11360\\
$(23,9,4)$ &&8855&32112&8465\\
$(23,11,5)$ &&&58786&17055\\
$(24,9,4)$ &&10626&38561&10796\\
$(24,12,5)$ &&&172159&53945\\
$(25,9,4)$ &&12650&46099&13720\\
$(26,13,6)$ &&230230&453169&101494\\
$(27,13,6)$ &&296010&742900&163216\\
\hline
\end{tabular}
\medskip
\caption{Bounds for the independence number of $J(n,w,i)$}
\label{Table1}
\end{table}

\begin{table}
\centering
\bigskip
\bigskip

\rotatebox{90}{
\begin{tabular}{|r|r|r|r|r|r|r|r|r|}
n& $G$ & $\alpha$ & $M$ & $r$ & $z_0$ & $z_1$ & $z_2$ &
$z_0+z_2\alpha/M$ \\
\hline
4& $600$-cell & 26 & 120 & $\sqrt{3}/3$ & 0.0421343 & 0.690511 &0.267355 & 0.100062\\
5& $600$-cell & 26 & 120 & $\sqrt{3}/3$ & 0.023477 & 0.772059 & 0.204465 & 0.0677778 \\
6& $600$-cell & 26 & 120 & $\sqrt{3}/3$ & 0.0141514 &0.830343 & 0.155506 & 0.0478444\\
7& $E_8$ kissing & 7 & 56 & $\sqrt{6}/4$ &  0.007948 & 0.834435 &0.157617  & 0.0276502\\
8 & $E_8$  &  36 & 240 & $\sqrt{3}/3$ &0.0053364&0.899613 & 0.0950508& 0.0195941\\
9 & $E_8$  &  36 & 240 & $\sqrt{3}/3$&0.0033303&0.921154&0.0755157&0.0146577\\
10 & $E_8$  &  36 & 240 & $\sqrt{3}/3$ &0.00209416&0.937453&0.0604529& 0.0111621 \\
11 & $E_8$  &  36 & 240 & $\sqrt{3}/3$ & 0.00132364&0.949973&0.0487036&0.00862918 \\
12 & $J(13,6,2)$ &148 &1716&$\sqrt{21/52}$&9.002e-04&0.938681& 0.0604188&0.00611112\\
13 & $J(14,7,3)$ & 184&3432&$\sqrt{7/16}$&5.933e-04& 0.936921&0.0624857&0.00394335\\
14& $J(15,7,3)$ &261 &6435&$\sqrt{7/15}$&3.9393e-04&0.935283&0.0643239&0.00300288\\
15 & $J(16,8,3)$ &850 &12870&$\sqrt{2/5}$&2.7212e-04&0.967168&0.0325604&0.00242258\\
16 & $J(17,8,3)$ & 1090&24310&$\sqrt{36/85}$&1.9080e-04&0.968014&0.0317961&0.00161646\\
17 & $J(18,9,4)$ & 1460&48620&$\sqrt{9/20}$&1.34658e-04&0.967557&0.0323093&0.00110487\\
18 & $J(19,9,4)$ & 2127&92378&$\sqrt{9/19}$&9.50746e-05&0.96714&0.032765&8.49488e-04\\
19 & $J(20,9,3)$ & 6708&167960&$\sqrt{33/80}$&5.944e-05&0.98275&0.0171908&7.46008e-04\\
20 & $J(21,10,4)$ &8639 &352716&$\sqrt{55/126}$&4.44363e-05&0.982618&0.0173381&4.69095e-04\\
21& $J(22,11,5)$ & 11360&705432&$\sqrt{11/24}$&3.2936e-05&0.982495&0.0174727&3.1431e-04\\
22& $J(23,11,5)$ & 17055&1352078&$\sqrt{11/23}$&2.4315e-05&0.982385&0.0175913&2.46211e-04\\
23& $J(24,12,5)$ & 53945&2704156&$\sqrt{3/7}$&1.40898e-05&0.990052&0.00993429&2.12269e-04\\
24& $\Omega(n)$ & 183373&$2^{24}$&$\sqrt{1/2}$&1.30001e-05&0.984309&0.0156786&1.84366e-04\\
\hline
\end{tabular}
\medskip
}
\caption{Feasible solutions of \eqref{LP} and corresponding upper bounds
  for $m_1(\R^n)$}
\label{Table2}
\end{table}

\begin{table}
\centering
\begin{tabular}{|r|r|r|}
n & previous best  & new \\
 & lower bound  & lower bound \\
 & for $\chi_m(\R^n)$ & for $\chi_m(\R^n)$ \\
\hline
4 & 9 \cite{OV}& 10\\
5& 14\cite{OV}&15\\
6&20 \cite{OV}&21\\
7&28 \cite{OV}&37\\
8& 39\cite{OV}&52\\
9& 54\cite{OV}&69\\
10&73 \cite{OV}& 90\\
11&97 \cite{OV}&116\\
12 &129 \cite{OV}& 164\\
13 & 168\cite{OV}& 254\\
14 & 217\cite{OV}& 334\\
15 & 279\cite{OV}& 413\\
16 &355 \cite{OV}& 619\\
17 & 448\cite{OV}& 906\\
18 & 563\cite{OV}& 1178\\
19 &705 \cite{OV}& 1341\\
20 & 879\cite{OV}& 2132\\
21& 1093\cite{OV}& 3182\\
22& 1359\cite{OV}& 4062\\
23&1690 \cite{OV}& 4712\\
24& 2106\cite{OV}& 5424\\
\hline
\end{tabular}
\medskip
\caption{Lower bounds for the measurable chromatic number}
\label{Table3}
\end{table}

\section{Tightening the theta number of compact graphs with
  subgraphs}\label{section-theta}

In this section, we would like  to show that the method presented in section \ref{section-LP} to upper bound
$m_1(\R^n)$ is flexible enough to be adapted to a broad variety of
situations, in order to design tight upper bounds for the
independence number of a graph. In fact, this method represents an
interesting way to strengthen the upper bound given by Lov\'asz  theta number of a graph, by exploiting an additional
constraint arising from a subgraph. 

The framework in which we will develop the method is that of a graph
$\G=(X,E)$ where $X$  is a compact topological  space, endowed with the
continuous and transitive action of a compact topological group $\Gamma$ ($X$ is
called a homogeneous space). This framework includes the case of
finite graphs where the vertex set is given the discrete topology.
 There are two reasons  why we do not limit ourselves
to the finite case: one is that going from finite graphs to compact
graphs does not raise essential difficulties; 
the other reason is that the compact case includes  spaces of special interest to us, 
in particular that  of the unit sphere $S^{n-1}$ (see Remark \ref{rem-sphere}). 

Before we dive into this rather general framework, we review the
 theta number of a finite graph.

\subsection{The theta number of a finite graph}
This number, denoted 
$\vartheta(\G)$ and  introduced in \cite{Lov},
is the optimal value of a semidefinite program that satisfies
\begin{equation*}
\alpha(\G)\leq \vartheta(\G)\leq \chi(\overline{\G})
\end{equation*}
where $\alpha(\G)$ denotes as before the independence number of $\G$,
$\overline{\G}$ is the complementary graph, and $\chi(\overline{\G})$
is its chromatic number, the least number of colors needed to color
all vertices so that adjacent vertices receive different colors.

One of the many equivalent definitions of $\vartheta(\G)$  involves matrices  $S$ 
whose rows and columns are indexed by $X$; for such a matrix, whose coefficients will be denoted
$S(x,y)$, we write $S\succeq 0$ if $S$ is symmetric and  positive
semidefinite. Then, following \cite[Theorem 4]{Lov}:
\begin{equation}\label{theta-general}
\begin{array}{cl}
 \vartheta(\G)=\sup\big\{ \sum_{(x,y)\in X^2} S(x,y)\ :\
& S\in \R^{X\times X},\ S\succeq 0, \\
& \sum_{x\in X} S(x,x)=1,\\
& S(x,y)=0\ (\{x,y\}\in E)\ \big\}.
\end{array}
\end{equation}
The inequality:
\begin{equation*}
\alpha(\G)\leq \vartheta(\G)
\end{equation*}
follows from the properties  of a certain 
matrix naturally associated to a subset  set $A\subset X$:
\begin{equation*}
S_A(x,y):=\1_A(x)\1_A(y)/|A|.
\end{equation*}
This matrix  satisfies a number of linear conditions:
\begin{equation*}
\sum_{x\in X} S_A(x,x)=1,\quad \sum_{(x,y)\in X^2} S_A(x,y)=|A|,
\end{equation*}and, if $A$ is an independent set, $S_A(x,y)=0$ for all $\{x,y\}\in E
$. Moreover, $S_A$ is positive semidefinite, so it defines a feasible solution of the semidefinite program in \eqref{theta-general}, with objective value equal to $|A|$.
The inequality $\alpha(\G)\leq \vartheta(\G)$ follows immediately. 

\begin{remark}\label{rem} In order to tighten the inequality $\alpha(G)\leq
 \vartheta(G)$ for finite graphs, it is customary to add the condition $S\geq 0$ (meaning all
  coefficients of $S$ are non negative) to the constraints in \eqref{theta-general}; the new
  optimal value is denoted $\vartheta'(G)$ and coincides with 
   the linear programming bound introduced earlier by P. Delsarte for the cardinality of codes in polynomial association
   schemes (see \cite{Del1} and \cite{Sch1}). 
\end{remark}

\subsection{Compact homogeneous graphs}
From now on, we assume that $X$ is a compact space, acted upon by a compact
topological group $\Gamma$ which is a
subgroup of the automorphism group of the graph $\G$. We assume that the application $(\gamma, x)\mapsto \gamma x$
defining the action of $\Gamma$ on $X$  is continuous, and that this action is transitive. 
We choose a base point $p\in X$, and let
$\Gamma_p$ denote the stabilizer of $p$ in $\Gamma$, so that $X$ can
be identified with the quotient space $\Gamma/\Gamma_p$ (see \cite[section 2.6]{Fol}). 

The group $\Gamma$ is equipped with its Haar  measure (see \cite[section 2.2]{Fol}), normalized so that its total 
volume equals $1$, which  induces a Borel regular measure on $X$, such
that for any measurable function $\varphi$ on $X$,
\begin{equation*}
\int_X \varphi(x) dx=\int_{\Gamma} \varphi(\gamma
p)d\gamma.
\end{equation*}
(see \cite[section 2.6]{Fol}). Volumes for this measure will be denoted $\vol_X$. The {\em independence
volume}  $\alpha_X(\G)$ of $\G$ is by definition the supremum of the volume of a measurable independent
set of $X$:
\begin{equation*}
\alpha_X(\G):=\sup  \{ \vol_X(A)\ :\ A\subset V, A\text{ is Borel measurable and independent }\}.
\end{equation*}
We will assume from now on that $\alpha_X(\G)>0$. 
We note that, if $X$ is finite,  the measure induced on
$X$ is simply
\begin{equation*}
\int_X \varphi(x) dx=\frac{1}{|X|}\sum_{x\in X} \varphi(x).
\end{equation*}
In particular, if $A\subset X$, $\vol_X(A)=|A|/|X|$ so
$\alpha_X(\G)=\alpha(\G)/|X|$. 

Now let $V\subset X$ be a Borel measurable subset of $X$ together with a finite positive Borel
measure $\lambda$ on $V$, such that $0<\lambda(V)<+\infty$. 
We introduce as in section \ref{section-LP} the $\lambda$-independence 
number of the  subgraph $G$ induced on $V$ by $\G$:
\begin{equation*}
\alpha_{\lambda}(G):=\sup\{\lambda(A)\ :\ A\subset V, A\text{ a Borel measurable independent set }\}.
\end{equation*}

In order to define  $\vartheta_G(\G)$, we need to introduce 
positive definite functions on $X$. If  $f$ belongs to  the space $\CC(X)$ of real valued continuous 
functions on $X$, we say that $f$ is {\em positive definite} and denote $f\succeq 0$ if, for all $k\geq 1$, 
for any choice of $\gamma_1,\dots,\gamma_k\in \Gamma$, the matrix with coefficients 
$f(\gamma_j^{-1}\gamma_i p)$ is symmetric positive semidefinite. Because 
$\Gamma$ is compact and $f$ is continuous, this condition is equivalent to the property that 
the function $\gamma\to f(\gamma p)$ is a  function of positive type  on $\Gamma$ in the sense of
\cite[section 3.3]{Fol} (see also \cite{DLV}). 

We note that this notion coincides with the notion of positive definite functions that came into play in 
section \ref{section-LP}. Indeed, the situation of section \ref{section-LP} corresponds to the case of the additive group
$\Gamma=\R^n/L$, acting on itself by translations, with $p=0^n$.

Now we can state the main result of this section:

\begin{theorem} \label{Theorem 5}With the notations introduced above, let 
\begin{equation}\label{theta-G}
\begin{array}{cl}
\vartheta_G(\G)=\sup\big\{ \int_X f(x) dx\ :\
& f\in \CC(X),\  f(\gamma x)=f(x) \ (\gamma\in \Gamma_p), \\
& f\succeq 0, \\
& f(p)=1,\\
& f(x)=0\ (\{x,p\}\in E)\\
& \int_V f(v)d\lambda(v) \leq \alpha_{\lambda}(G)\ \big\}.
\end{array}
\end{equation}
Assuming $\alpha_X(\G)>0$, we have
\begin{equation*}
\alpha_X(\G)\leq \vartheta_G(\G).
\end{equation*}
\end{theorem}

\begin{proof} 
Let $A\subset X$ be a Borel measurable independent set of positive measure. We introduce 
\begin{equation*}
f_A(x):= \frac{1}{\vol_X(A)}\int_{\Gamma} \1_A(\gamma x)\1_A(\gamma p)d\gamma.
\end{equation*}
This function $f_A\in \R^X$ will play the role of the matrix $S_A$ that
occurred  in the proof
of the inequality $\alpha(\G)\leq \vartheta(\G)$ for finite graphs. We
claim that $f_A$ satisfies the constraints required by the program
defining $\vartheta_G(\G)$.  Indeed, being the convolution over
$\Gamma$ of two
bounded functions, $f_A$ is continuous (see \cite[(2.39)]{Fol}). The
other conditions (numbered (i) to (v) in order of appearance in \eqref{theta-G}) are easily obtained, applying Fubini's theorem to
swap integrals and the invariance by left and right multiplication of
the Haar measure on $\Gamma$. We skip the details for (i) and
(iii). Condition (iv) holds because on one hand, if $\{x,p\}\in E$, also $\{\gamma x,\gamma p\}\in E$, and on the other hand, 
 $A$ is an independent set of $\G$,  so $\1_A(\gamma x)\1_A(\gamma p)=0$.

Let us check (ii), i.e. $f_A\succeq 0$: thanks to the right invariance of the Haar measure, we have
\begin{align*}
 f_A(\gamma_j^{-1}\gamma_i p)&= \frac{1}{\vol_X(A)}\int_{\Gamma}\1_A(\gamma\gamma_j^{-1}\gamma_i p)\1_A(\gamma p)d\gamma\\
&=\frac{1}{\vol_X(A)}\int_{\Gamma}\1_A(\gamma\gamma_i p)\1_A(\gamma\gamma_j p)d\gamma.
\end{align*}
So, the matrix with coefficients $f_A(\gamma_j^{-1}\gamma_i p)$ is
symmetric. Moreover, for $(x_1,\dots,x_k)\in \R^k$, 
\begin{equation*}
\sum_{1\leq i,j\leq k} x_i x_j f_A(\gamma_j^{-1}\gamma_i p)=
\frac{1}{\vol_X(A)}\int_{\Gamma}\Big(\sum_{i=1}^k x_i \1_A(\gamma\gamma_i p)\Big)^2d\gamma \geq 0.
\end{equation*}

In order to verify (v), we remark that
\begin{equation}\label{ineq1}
\int_{V} \1_A(\gamma v)d\lambda(v)=\lambda((\gamma^{-1}A)\cap V)\leq \alpha_{\lambda}(G).
\end{equation}
This inequality, combined with Fubini's theorem, leads to the
result. Indeed,
\begin{align*}
\int_V f_A(v)d\lambda(v) &= \frac{1}{\vol_X(A)} \int_V \int_{\Gamma}  \1_A(\gamma v)\1_A(\gamma p) d\gamma d\lambda(v)\\
&= \frac{1}{\vol_X(A)} \int_{\Gamma} \Big(\int_V  \1_A(\gamma v)
d\lambda(v)\Big)\1_A(\gamma p) d\gamma \\
&\leq  \frac{1}{\vol_X(A)} \int_{\Gamma}
\alpha_{\lambda}(G)\1_A(\gamma p) d\gamma =\alpha_{\lambda}(G).
\end{align*}

It remains to compute the objective value or $f_A(x)$; for this, we
apply Fubini's theorem once more:
\begin{align*}
\int_X f_A(x) dx &=\int_X \frac{1}{\vol_X(A)} \int_{\Gamma}
\1_A(\gamma x)\1_A(\gamma p) d\gamma dx\\
&=\frac{1}{\vol_X(A)} \int_{\Gamma} \Big(\int_X \1_A(\gamma x)
dx\Big)\1_A(\gamma p) d\gamma \\
&=\frac{1}{\vol_X(A)} \int_{\Gamma} \vol_X(\gamma^{-1} A) \1_A(\gamma
p) d\gamma \\
&=\int_{\Gamma} \1_A(\gamma p) d\gamma =\vol_X(A).
\end{align*}

\end{proof}

\begin{remark}\label{rem-sphere} Taking $X=S^{n-1}$, and $E=\{\{x,y\}\ :\ \Vert
x-y\Vert =d\}$, defines a graph $G(S^{n-1},d)$ homogeneous under the action of the
orthogonal group that fits into our setting. Moreover, up to a
suitable rescaling, this graph is an induced subgraph of  the unit
distance graph. So,  Theorem \ref{Theorem 5} can be  applied to compute tight bounds 
for $\alpha(G(S^{n-1},d))$, which in turn may be used in
Theorem \ref{theorem-general}, suggesting an inductive
method to calculate better upper bounds for $m_1(\R^n)$. 

\end{remark}

In the remaining of this section, we discuss some connections between Theorem \ref{Theorem 5} and 
previous results.

\subsection{The relationship between $\vartheta_G(\G)$ and $\vartheta(\G)$ for finite homogeneous graphs.}

If $\G$ is a finite homogeneous graph for the group $\Gamma$, its theta number can be rewritten as:
\begin{equation*}
\begin{array}{cl}
\vartheta(\G)=\sup\big\{ \sum_{x\in X} f(x) \ :\
& f\in \R^X,\ f(\gamma x)=f(x) \ (\gamma\in \Gamma_p), \\
& f\succeq 0, \\
& f(p)=1,\\
& f(x)=0\ (\{x,p\}\in E)\ \big\}.\\
\end{array}
\end{equation*}
We refer to \cite[Theorem 2]{DLV} where this reformulation is given in
the  special case of  Cayley graphs. 
The generalization to homogeneous graphs is straightforward. 
So, $\vartheta_G(\G)$ is a tightening of $\vartheta(\G)$ with an additional constraint relative to the subgraph $G$, and we have
\begin{equation*}
\vartheta_G(\G)\leq \vartheta(\G).
\end{equation*}

\subsection{The analogy between $\vartheta_G(\G)$ and $\vartheta_G(\R^n)$.}

Our notations suggest an analogy between $\vartheta_G(\G)$ as defined in \eqref{theta-G}
and $\vartheta_G(\R^n)$ as introduced in Theorem \ref{theorem-general}. This analogy will be more transparent
from the expression of the dual program of $\vartheta_G(\G)$. Here we apply the duality theory of conic linear programs 
in locally convex topological vector spaces for which we refer to
\cite[Chapter IV]{Bar}. The dual space of the space $\CC(X)$ of real valued continuous
functions on $X$, i.e. the space of continuous linear forms on
$\CC(X)$ equipped with the topology defined by the supremum norm,   is the space $\M(X)$ of signed regular measures on $X$ 
(it follows from Riesz representation theorem, see 
\cite[Theorem 6.19]{Rud}). For $\mu\in \M(X)$, the notation $\mu\succeq 0$ ($\mu$ positive definite) stands for $\int fd\mu\geq 0$ 
for all $f\in \CC(X)$, $f$ being positive definite. The support of
$\mu$ is denoted by $\Supp(\mu)$. The dual program of $\vartheta_G(\G)$ in the sense of \cite[Chapter 4, section 6]{Bar} becomes:
\begin{equation*}
\begin{array}{cl}
\inf\big\{ z_0+z_2\frac{\alpha_{\lambda}(G)}{\lambda(V)}\ :\ 
& \mu\in \M(X), \ \Supp(\mu)\subset\{x\in X\ :\ \{x,p\}\in E\}\\
& z_2\geq 0\\
& z_0\delta_p + \mu +z_2\frac{\lambda}{\lambda(V)} - dx\succeq 0 \ \big\}
\end{array}
\end{equation*}
We recall that weak duality holds, i.e. that $\vartheta_G(\G)$ is upper bounded by the optimal value of its dual program
(see \cite[Theorem 6.2]{Bar}).

\subsection{An inequality relating $\alpha_X(\G)$ and $\alpha_{\lambda}(G)$ and its connection to $\vartheta_G(\G)$}
Let us go back to the inequality \eqref{ineq1}. If we
  integrate it over $\Gamma$, and then take the supremum over the independent
  sets of $\G$, we obtain 
\begin{equation*}
\alpha_X(\G)\leq \frac{\alpha_{\lambda}(G)}{\lambda(V)}.
\end{equation*}
In particular, if $\G$ is a finite graph and $\lambda$ is the counting
measure, the above inequality becomes
\begin{equation}\label{ineq-ratios}
\frac{\alpha(\G)}{|X|}\leq \frac{\alpha(G)}{|V|}.
\end{equation}
We recover  a standard inequality that has proved to be useful in several instances, in particular if one
has a special hint on $G$. For example, in coding theory it is applied to relate the 
sizes of codes in Hamming and Johnson spaces respectively, following  Elias and
Bassalygo principle (see e.g. \cite{MRRW}). Also, Larman and Rogers inequality
\eqref{Larman-Rogers} can be seen as an analogue of \eqref{ineq-ratios} for the
unit distance graph.

\section{Open problems}

We present here a few questions that we believe would be worth to look at. Some of them have already been 
mentioned previously.

\begin{enumerate}
\item There are several possible variants in the way we apply Theorem \ref{theorem-general} to find upper bounds 
for $m_1(\R^n)$. There is no reason to restrict to graphs that embed in a sphere centered at $0^n$ as we do, 
and also several subgraphs could be used simultaneously. Can the bounds of Tables 2 and 3 be improved this way ?

\item The subgraph method can be applied to strengthen the theta number of finite graphs, in particular it could be used
to obtain better bounds for the independence number of the graphs $J(n,w,i)$. In turn, the resulting upper bounds 
may lead to further improvements on the bounds for $m_1(\R^n)$.  More generally, can Theorem \ref{Theorem 5} applied to the unit sphere  lead to
improved bounds for $m_1(\R^n)$ ( see Remark \ref{rem-sphere})?

\item In coding theory, the so-called MRRW-bound \cite{MRRW} is an asymptotic upper bound for the size of codes 
of given minimal Hamming distance, which derives from Delsarte linear programming bound  
(Lov\'asz theta number provides essentially the same bound). 
For binary codes and for a certain range of minimal distances, it is superseded by the so-called second MRRW-bound, 
which is obtained from the inequality \eqref{ineq-ratios}, where $X$ is the Hamming space and $V$ is a Johnson space with suitable weight
(i.e. the set of binary words of fixed weight). Again, Delsarte linear programming bound is applied to $V$. 

Following Theorem \ref{Theorem 5}, it is possible to design a program that combines the two bounds in one. Is it possible
to improve the MRRW bounds by analyzing this bound ?

\end{enumerate}

\section*{Acknowledgements}

We thank Fernando Oliveira and Frank Vallentin for helpful discussions,
and Andrei Raigorodskii for pointing out 
to us the reference \cite{AK}. We thank the reviewers for their careful reading and for  their remarks  that have helped to improve the presentation of this work.

\end{document}